%% file: C2IntBGSPArxiv2.tex
\documentclass[]{amsart}  %Use for Arxiv
\input{preamble}

\title{A Thom Spectrum Model for $C_2$-Integral Brown--Gitler Spectra}

\author{Guchuan Li \and Sarah Petersen \and Elizabeth Tatum}

\begin{document}
	
	\begin{abstract}
	A Thom spectrum model for a $C_2$-equivariant 
        analogue of integral Brown--Gitler spectra is established and shown to have a multiplicative property. The $C_2$-equivariant spectra constructed enjoy properties analogous to classical nonequivariant integral Brown--Gitler spectra and thus may prove useful for producing $C_2$-equivariant analogues of splittings of $BP \langle 1 \rangle \wedge BP \langle 1 \rangle$ and $bo \wedge bo.$
	\end{abstract} 
	
	\maketitle
        \tableofcontents
	
\section{Introduction}

In the 1970's, Brown and Gitler constructed a family of spectra realizing certain sub-comodules of the dual Steenrod algebra at the prime $p=2$ \cite{brown1973spectrum}. Brown--Gitler's motivation for constructing the original spectra was to study immersions of manifolds. They have also been useful for studying maps out of classifying spaces. For example, they were used by Miller to prove the Sullivan Conjecture \cite{miller1984sullivan}. 
  
Brown--Gitler used an obstruction-theoretic approach, first constructing an algebraic resolution, then constructing a tower of spectra realizing that resolution, and then taking its inverse limit to obtain the desired spectrum. In \cite{Mahowald1977}, Mahowald suggested an alternative construction using a Thom spectrum model. In particular, Mahowald posited that the natural filtration of the space $\Omega^2 S^3$ due to May--Milgram \cite{MilMay81} produces a filtration of the Thom spectrum $(\Omega^2 S^3)^\mu \simeq H {\mF}_2$ by spectra. Here, $\mu$ is the double loop map 
$\mu: \Omega^2 S^3 \to BO$ of the classifying map of the M\"{o}bius bundle over $S^1.$ The resulting filtration spectra turn out to be the Brown--Gitler spectra of \cite{brown1973spectrum} (see \cite{BP78,cohen1979geometry,HuntKuhn99}). In \cite{cohen1979geometry}, Cohen made this construction precise at all primes. 

In \cite{mahowald1981bo}, Mahowald suggested that integral analogues of Brown--Gitler spectra should exist: that is, finite spectra realizing an analogous family of sub-comodules of $H_{*}H\mathbb{Z}$, and proposed a Thom spectrum model. In \cite{shimamoto1984integral}, Shimamoto gave a construction using the obstruction-theoretic approach, and also gave more details on Mahowald's proposed Thom spectrum model. Kane used an odd primary version of this Thom spectrum construction in \cite{kane1981operations}. In \cite{CDGM1988}, Cohen, Davis, Goerss, and Mahowald made these Thom spectrum constructions rigorous, in particular by explicitly defining the base space used in the Thom spectrum construction. 

In \cite{goerss1986some}, Goerss--Jones--Mahowald used the obstruction-theoretic approach to construct Brown--Gitler spectra for $BP\langle 1 \rangle$ and $bo$ (i.e. finite spectra realizing certain sub-comodules of $H_{*}BP\langle 1 \rangle$ and $H_{*}bo$), and Klippenstein extended these arguments to the $BP\langle 2 \rangle$ case in \cite{klippenstein1988brown}. These spectra are collectively known as generalized Brown--Gitler spectra. 

The higher Brown--Gitler spectra have been useful in understanding the smash products of various ring spectra. In particular, Mahowald used the integral Brown--Gitler spectra to decompose $bo \wedge bo$ as a sum of finitely generated $bo$-modules \cite{mahowald1981bo}. Kane subsequently produced an analogous decomposition for $BP\langle 1 \rangle$, also using integral Brown--Gitler spectra \cite{kane1981operations}. These splittings helped make it feasible to use the $bo$- and $BP\langle 1\rangle$-based Adams spectral sequences for computations and have proved to be powerful computational tools. For example, Mahowald used the $bo$-based Adams spectral sequence to prove the telescope conjecture at the prime $2$ and height $1$ \cite{mahowald1981bo}.

This paper extends Mahowald's idea for an integral analogue of Brown--Gitler spectra to the $C_2$-equivariant setting (Theorem \ref{thm:C2intBGspectra}). The $C_2$-equivariant spectra we construct realize an analogous family of sub-comodules of $H {\umF_2}_\star H \umZ,$ where $H \umF_2$ is the Eilenberg-MacLane spectrum associated to the constant $C_2$-Mackey functor $\umF_2$ and $H \umZ$ is the Eilenberg-MacLane spectrum associated to the constant $C_2$-Mackey functor $\umZ$. Our primary motivation for establishing a construction of these spectra is to enable the study of $C_2$-equivariant analogues of splittings of $BP \langle 1 \rangle \wedge BP \langle 1 \rangle$ and $bo \wedge bo,$ and related $BP \langle 1 \rangle$- and $bo$- based spectral sequence computations. 

In the $C_2$-equivariant setting, the Eilenberg-MacLane spectra $H \umF_p,$ $H \umZ_{(2)},$ and $H \umZ_2,$ associated to the constant $C_2$-Mackey functors $\umF_2,$ $\umZ_{(2)},$ and $\umZ_2$ respectively, also arise from Thom spectra. In \cite{BW2018}, Behrens--Wilson show there is an equivalence of $C_2$-spectra $(\Omega^\rho S^{\rho + 1})^\mu \simeq H \umF_2,$ where $\rho$ is the regular representation of $C_2$ and $\mu$ is the $\rho$-loop map
$ \mu: \Omega^\rho S^{\rho + 1} \to B_{C_2} O $
of the classifying map of the M\"{o}bius bundle over $S^1$ regarded as a $C_2$-equivariant virtual bundle of dimension zero by endowing both $S^1$ and the bundle with trivial action. Note we follow the convention that $\mZ_{(2)}$ denotes the $2$-local integers and $\mZ_2$ the $2$-adic integers.

In \cite{HW2020}, Hahn--Wilson generalize the result of Behrens--Wilson. They show the $G$-equivariant mod $p$ Eilenberg-MacLane spectrum arises as an equivariant Thom spectrum for any finite, $p$-power cyclic group $G.$ Hahn--Wilson also establish a Thom spectrum model for $H \umZ_{(p)},$ building a base space from  the $G$-space $\Omega^\lambda S^{\lambda + 1},$ where $\lambda$ denotes the standard representation of $G$ on the complex numbers with a generator acting by $e^{2\pi i / p^n}.$ They observe that this space carries the arity filtration from the $E_\lambda$-operad and thus one could define equivariant Brown--Gitler spectra as the spectra coming from this filtration. This is the perspective we take and extend to the integral setting in this paper. 

Non-equivariantly, two conditions characterize Brown--Gitler spectra. First, these spectra realize certain sub-comodules of the dual Steenrod algebra. Additionally, they satisfy a surjectivity condition coming from the geometry involved in Brown--Gitler's original construction \cite{brown1973spectrum}. While the Thom spectrum model does satisfy this additional condition in the non-equivariant setting, it is not clear from neither the geometry nor the Thom spectra filtration models, 
how to generalize this second condition in the equivariant setting. We discuss this in Remark \ref{rmk:GeomCond}, taking the philosophy proposed by Hahn--Wilson: one can simply define equivariant Brown--Gitler spectra using filtrations on the Thom spectra models and see if these spectra are computationally useful. 

Several factors motivate our choice to study $C_2$-equivariant, as opposed to $C_p$-equivariant spectra for all primes $p$. First, $RO(C_2)$-graded $H \umF_2$-homology is free over its coefficients in important examples such as the dual Steenrod algebra \\ ${H \umF_2}_\star H \umF_2$ \cite{HuKriz2001} and ${H \umF_2}_\star H \umZ$ \cite{Ormsby2011}. This is helpful because integral Brown--Gitler spectra topologically realize certain sub-comodules of ${H \umF_2}_\star H \umZ$. In contrast, $RO(C_p)$-graded $H \umF_p$-homology is often not free over its coefficients when $p > 2.$ In particular, the $C_p$-dual Steenrod algebra is not free \cite{SanW2il022}, which suggests the odd primary $C_p$-equivariant story may be more complicated, requiring techniques beyond those developed in this paper. 

Furthermore, when $G = C_2,$ there is a non-obvious $C_2$-equivalence of spaces $\Omega^\lambda S^{\lambda + 1} \simeq \Omega^\rho S^{\rho + 1}$ \cite{HW2020}. Thus there are filtrations of $\Omega^\lambda S^{\lambda + 1} \simeq \Omega^\rho S^{\rho + 1}$ by both the $E_\lambda$ and $E_\rho$ operads, leading to at least two definitions of Brown--Gitler spectra.  In this paper, we choose to work with the $E_\rho$-filtration because it is the most computationally accessible. Studying these two definitions of Brown--Gitler spectra may prove an interesting direction for future work.

\section{Statement of Theorems} Our main result is a $C_2$-equivariant analogue of \cite[Theorem~1.5(i),(ii)]{CDGM1988}. To state this theorem precisely, we recall ${H \umF_2}$ has distinguished elements $a \in H {\umF_2}_{\{-\sigma\}}$ and $u \in H {\umF_2}_{\{1 - \sigma\}},$ where $\sigma$ is the one-dimensional sign representation of $C_2$. We 
 define a weight filtration on 
$$ {H \umF_2}_\star H {\umZ} \cong H {\umF_2}_\star [\bar{\xi_1}, \bar{\xi_2}, \bar{\xi_3}, \cdots, c(\tau_1), c(\tau_2), \cdots ]/ (c(\tau_i^2) = a c(\tau_{i + 1}) + u \bar{\xi}_{i +1}),$$
	where $\vert \tau_j \vert = 2^j \rho - \sigma,$ $\vert \bar{\xi}_i \vert = (2^i - 1) \rho,$ and $c$ denotes the antiautomorphism of the dual Steenrod algebra $\cA \cong \pi_\star H \umF_2 \wedge H \umF_2$ (the computation of $H {\umF_2}_\star H \umZ$ follows from \cite[Theorem~3.8]{Ormsby2011}). A weight filtration is defined by 
	$$\wt(c(\tau_j)) = \wt(\bxi_j) = 2^j, \qquad \wt(xy) = \wt(x) + \wt(y). $$

\begin{thm*}[Theorem \ref{thm:C2intBGspectra}]
	For $n > 0,$ there is a $2$-complete spectrum $B_0(n)$ and a map 
	$$B_0(n) \xrightarrow[]{g} H {\umZ_{2}} $$
		such that
		\begin{itemize}
			\item[(i)] $g_\star$ sends $H {\umF_2}_\star(B_0(n))$ isomorphically onto the span of monomials of weight $\leq 2n;$
			\item[(ii)] there are pairings
			$$ B_0(m) \wedge B_0(n) \to B_0 (m + n) $$
			whose homology homomorphism is compatible with the multiplication in $H {\umF_2}_\star ( H \umZ_2 ).$
		\end{itemize}
	\end{thm*}

The spectra $B_0(n)$ are $C_2$-equivariant analogues of Cohen, Davis, Goerss and Mahowald's integral Brown--Gitler spectra in the sense that they realize certain sub-comodules of $H {\umF_2}_\star H \umZ.$ This result follows from the more technical Theorem \ref{thm:main}, which is a $C_2$-equivariant analogue of \cite[Theorem~1.3]{CDGM1988}. To state Theorem \ref{thm:main}, we introduce both an increasing filtration of the space $\Omega^\rho S^{\rho + 1}$ and a homotopy fiber sequence.

Work of Rourke-Sanderson \cite{RS200} shows the space $\Omega^\rho S^{\rho + 1}$ admits an increasing filtration by 
\begin{align} \label{eq:filtration}
    F_n \Omega^\rho S^{\rho + 1} \simeq \underset{0 \leq k < n}{\displaystyle{\coprod}} C_k (\rho) \underset{\Sigma_k}{\times} (S^1)^k/ \sim
\end{align}
where
\[
C_k (\rho) = \{m_1, m_2, \cdots, m_k | m_i \neq m_j \text{ if } i \ne j \}
\]
is the configuration space of $k$ ordered points in the $C_2$-regular representation $\rho.$ Note when $k = 0,$ this gives the base point $x_0.$ If $x_n = x_0,$ the relation $\sim$ identifies
\[
(m_1, \cdots, m_n; x_1, \cdots x_n) \sim (m_1, \cdots, m_{n -1}; x_1, \cdots, x_{n - 1}).
\]
\begin{rmk}
    The observant reader might notice our definition of $F_n \Omega^\rho S^{\rho + 1}$ differs slightly from that of \cite{RS200}. This stems from the fact that Rourke--Sanderson work entirely with fixed points. A description comparing definitions can be found following Theorem 1.11 in \cite{GuillouMay2017}.
\end{rmk}
Let $X_2$ denote the Bousfield localization of $X$ with respect to $H {\umF_2}$ and let $\cF_n = (F_n \Omega^\rho S^{\rho + 1})_2.$
Then there are product maps 
$$ \cF_m \times \cF_n \overset{\mu}{\rightarrow} \cF_{m + n} $$
induced by the corresponding maps for the filtration spaces and the fact that localization preserves finite products. Define $A_n$ by the homotopy fiber sequence 
\begin{align} \label{eq:fib}
	A_n \to \cF_{2n + 1} \to S^1_2
\end{align}
where the second map is the $H \umF_2$-localization of the composite
$$ F_{2n + 1} \Omega^\rho S^{\rho + 1} \to \Omega^\rho S^{\rho + 1} \to S^1. $$ 

\begin{thm*}[Theorem \ref{thm:main}]
    The fiber sequence (\ref{eq:fib}) is equivalent to a product fibration. Indeed, there is a $C_2$-equivariant map $A_n \overset{\phi}{\rightarrow} \cF_{2n}$ and a commutative diagram of fibrations
	\[
	\begin{tikzcd}
		A_n \arrow[rr] \arrow[d]{} & & A_n \arrow[d] \\
			S_2^1 \times A_n \arrow[r, "1 \times \phi"] \arrow[d, "p_1"] & S_2^1 \times \cF_{2n}= \cF_1 \times \cF_{2n} \arrow[r, "m"] & \cF_{2n + 1} \arrow[d] \\
			S_2^1 \arrow[rr, equal] & & S_2^1 
		\end{tikzcd}
		\]
		which is an equivalence on total spaces and on fibers.
		
\end{thm*}

Our argument for deducing Theorem \ref{thm:C2intBGspectra} from Theorem \ref{thm:main} follows that of Cohen--Davis--Goerss--Mahowald and thus relies on a Thom spectrum construction of $H \umZ_2$ with an $H \umF_2$-local base space (Theorem \ref{thm:HZcomp}).

In \cite{HW2020}, Hahn--Wilson establish a construction of $H \umZ_{(2)}$ as a $C_2$-Thom spectrum. In Section \ref{sec:ThomMod}, we extend Hahn--Wilson's arguments to construct $H \umZ_2$ as a $C_2$-equivariant Thom spectrum with an $H \umF_2$-local base space, resulting in Theorem \ref{thm:HZcomp}, an equivariant analogue of \cite[Theorem~1]{CoMayTay81}, which was originally proposed by Mahowald in the nonequivariant setting. 

To state Theorem \ref{thm:HZcomp}, we require the following maps. Consider the $\rho$-loops of the unit map $S^{\rho + 1} \to K(\umZ, \rho + 1).$ Composing with the adjoint to $-1 \in \pi_0^{C_2} (\mS_2^0)$, we get a map
    $$ \Omega^\rho S^{\rho + 1} \to \Omega^\rho K(\umZ, \rho + 1) \to BG L_1(\mS_2^0).$$
    Note $\Omega^\rho K(\umZ, \rho + 1) \simeq S^1$ with trivial $C_2$-action.

    Let $\Omega^\rho S^{\rho + 1} \langle \rho + 1 \rangle$ denote the homotopy fiber of the map $\Omega^\rho S^{\rho + 1} \to S^1$ and consider the composition 
    \[
    \mu: \Omega^\rho S^{\rho + 1} \langle \rho + 1 \rangle \to \Omega^\rho S^{\rho + 1}\to S^1 \to BGL_1(\mS_2^0). \label{eq:bundle}
    \]

    \begin{thm*}[\ref{thm:HZcomp}]
        There is an equivalence of $C_2$-spectra
    \begin{align*} 
        (\Omega^\rho S^{\rho + 1} \langle \rho + 1 \rangle_2)^\mu \to H \umZ_2. 
    \end{align*}
    \end{thm*}

\subsection*{Acknowledgments}

The authors would like to thank to Agn\`{e}s Beaudry, Tobias Barthel, Mark Behrens, Christian Carrick, David Chan, John Greenlees, Bert Guillou, Paul Goerss, and Jeremy Hahn for enlightening conversations. The first and second authors would also like to thank the Max Planck Institute for Mathematics in Bonn for its hospitality and financial support. The third author thanks the Hausdorff Institute in Bonn for its hospitality and financial support, as well as the Knut and Alice Wallenberg Foundation for financial support. All three authors are grateful to the Hausdorff Institute in Bonn for its hospitality during the Spectral Methods in Algebra, Geometry, and Topology trimester program in Fall of 2022, funded by the Deutsche Forschungsgemeinschaft (DFG, German Research Foundation) under Germany's Excellence Strategy- EXC-2047/1-390685813. This material is based upon work supported by the National Science Foundation under Grant No. DMS 2135884.

\subsection{Structure of Argument}

We extend the arguments of Cohen--Davis--Goerss--Mahowald to the $C_2$-equivariant setting. To make this extension clear, we first recall Cohen--Davis--Goerss--Mahowald's nonequivariant argument, then outline the structure of our equivariant extension. 

In \cite{CDGM1988}, Cohen--Davis--Goerss--Mahowald make a Thom spectrum model for integral Brown--Gitler spectra precise by explicitly defining the base space. Specifically, they consider the space $\Omega^2 S^3 \langle 3 \rangle,$ the double loop space of the homotopy fiber of the unit map
$ S^3 \to K(\mZ, 3)$. They show that the $H \umF_2$-localization $\Omega^2 S^3 \langle 3 \rangle_2$ carries a filtration which induces the weight filtration on homology. The filtration pieces of the space $\Omega^2 S^3 \langle 3 \rangle_2$ are the base spaces in the Thom spectrum model for integral Brown--Gitler spectra. 

To extend Cohen--Davis--Goerss--Mahowald's nonequivariant argument to the $C_2$-equivariant setting, we show that the space $\Omega^\rho S^{\rho + 1} \langle \rho + 1 \rangle_2$ carries a filtration which induces an analogous weight filtration on $H \umF_2$-homology. This requires a model of $H \umZ_2$ as a Thom spectrum with $H \umF_2$-local base space $\Omega^\rho S^{\rho + 1} \langle \rho + 1 \rangle_2.$

We provide equivariant preliminaries in Section \ref{sec:prelim}. We construct $H \umZ_2$ as a $C_2$-equivariant Thom spectrum with $H\umF_2$-local base space (Theorem \ref{thm:HZcomp}) in Section \ref{sec:ThomMod}. In Section \ref{sec:fib}, we extend Cohen--Davis--Goerss--Mahowald's technical argument making the filtration of the base space
rigorous. This culminates in the proof of Theorem \ref{thm:main}, from which our main result concerning a Thom spectrum model for $C_2$-integral Brown--Gitler spectra (Theorem \ref{thm:C2intBGspectra}) immediately follows. 

\section{Equivariant Preliminaries} \label{sec:prelim}

\subsection{${H \umF_2}$-Bousfield localization and $2$-completion} \label{subsec:Bousfield}

A $C_2$-spectrum $X$ is connective if both $X$ and $X^{C_2}$ are connective. Define the $2$-completion of $X$ as $\holim X/2^k$, where $X/2^k$ denotes the cofiber of $X \xrightarrow{2^k} X$. Similarly to the non-equivariant case, for a connective $C_2$-spectrum $X$,
\begin{equation}\label{equation:completion}
      X^\wedge_2 \simeq  L_{H\umF_2} X. 
\end{equation}

In this paper, we use the notation $X_{2}$ to denote the $H\umF_{2}$-localization of any space or spectrum $X$. Note that all spectra in this paper are connective, so the $H\umF_2$-localization of any spectrum $X$ will coincide with its $2$-completion.

We outline how to deduce Equation (\ref{equation:completion}) analogously to the nonequivariant argument. Let $\mathbb{S}/2$ denote the $C_2$-spectrum with trivial $C_2$-action. One can show that $X^\wedge_2 \simeq  L_{\mathbb{S}/2} X$ by following the proof of the nonequivariant version \cite[Proposition 2.5]{Bo1979}. 
Let $\underline{A}$ be the Burnside Mackey functor. Replacing $H\mathbb{Z}$ in \cite[Proposition 2.11]{Bo1979} by $H\underline{A}$, one can similarly show that $L_{H\underline{A}\wedge \mathbb{S}/2} X \simeq L_{\mathbb{S}/2}(L_{H\underline{A}} X)$. We use the connective condition here to note that $L_{H\underline{A}} X \simeq X$. Now we have $$X^\wedge_2 \simeq L_{\mathbb{S}/2} X \simeq L_{H\underline{A}/2} X .$$
Finally, recall that $\underline{A}(C_2/e)\cong \mathbb{Z}$ and $\underline{A}(C_2/C_2) \cong \mathbb{Z}[\alpha]/(\alpha^2-2\alpha)$. Then the ideals $(2)$ and $(2,\alpha)$ have the same radical and $\underline{A}/(2,\alpha) \cong \underline{\mathbb{F}_2}$. One can finish the proof of \ref{equation:completion} by showing that $L_{H\underline{A}/2} X \simeq L_{H\underline{A}/(2, \alpha)} X \simeq L_{H\umF_2} X$. In fact, one can also replace $2$ with any prime $p$ and Equation (\ref{equation:completion}) still holds.

\subsection{May--Milgram Filtration}
As in \cite[Section 4]{BW2018}, we will make use of an identification of the $C_2$-fixed points of the space $\Omega^\rho S^{\rho + 1}.$ Consider the cofiber sequence
$${C_2}_+ \to S^0 \hookrightarrow S^\sigma.$$
Mapping out of this cofiber sequence gives a fiber sequence
$$ \Omega N^\times \Omega S^{\rho + 1} \to \Omega^\rho S^{\rho + 1} \to \Omega S^{\rho + 1} \xrightarrow{\Delta} N^\times \Omega S^{\rho + 1},$$
where $N^\times X : = \text{Map}(C_2, X) = X \underset{\substack{ \curvearrowbotleftright \\ C_2}}{\times} X$ is the norm with respect to Cartesian product (i.e. the coinduced space). On taking fixed points, we get a fiber sequence
$$\Omega^{2} S^3 \xrightarrow[]{} (\Omega^\rho S^{\rho + 1})^{C_2} \to \Omega S^2 \xrightarrow[]{\text{null}} \Omega S^3. $$
In particular, there is an equivalence
\begin{align} \label{eq:fixedpts}
    (\Omega^\rho S^{\rho + 1})^{C_2} \simeq \Omega S^2 \times \Omega^2 S^3.
\end{align}
	
Behrens--Wilson \cite{BW2018} also established an additive isomorphism
	\begin{align} \label{eq:hOmega2S3}
		H {\umF_2}_\star \Omega^\rho S^{\rho + 1} \cong H {\umF_2}_\star \otimes E [t_0, t_1, \cdots ] \otimes P [ e_1, e_2, \cdots ] 
	\end{align}
	where
	\begin{align*}
		\vert t_i \vert & = 2^i \rho - \sigma \\
		\vert e_i \vert & = (2^i - 1) \rho.
	\end{align*}
We define a weight on the monomials in $H {\umF_2}_\star (\Omega^\rho S^{\rho + 1})$ by 
	$$\wt(t_j) = \wt(e_j) = 2^j, \qquad \wt(ab) = \wt(a) + \wt(b), $$
and recall the space $\Omega^\rho S^{\rho + 1}$ admits an increasing filtration by spaces
\begin{align*}
   F_n \Omega^\rho S^{\rho + 1} \simeq \underset{0 \leq k < n}{\displaystyle{\coprod}} C_k(\rho) \underset{\Sigma_k}{\times} (S^1)^{\times k} / \sim, 
\end{align*}
where the relation is defined in equation (\ref{eq:filtration}).

\begin{proposition}
    The filtration $F_n \Omega^\rho S^{\rho + 1}$ is such that $H {\umF_2}_\star (F_n \Omega^\rho S^{\rho + 1} )$ is the span of monomials of weight $\leq n.$
\end{proposition}

\begin{proof}
 Let $F_n H {\umF_2}_\star \Omega^\rho S^{\rho + 1}$ denote the span of monomials of $H {\umF_2}_\star \Omega^\rho S^{\rho + 1}$ of weight less than or equal to $n$ and consider $t^\epsilon e^k := t_0^{\epsilon_0} t_1^{\epsilon_1} \cdots e_1^{k_1} e_2^{k_2} \in F_n H {\umF_2}_\star \Omega^\rho S^{\rho + 1}.$ Similarly to the nonequivariant case \cite[p. 239]{CohenLadaMay76}, the inclusion \\ $F_n H{\umF_2}_\star \Omega^\rho S^{\rho + 1} \hookrightarrow H {\umF_2}_\star \Omega^\rho S^{\rho + 1}$ factors through $H {\umF_2}_\star F_n \Omega^\rho S^{\rho + 1}$ since every monomial $ t^\epsilon e^k \in F_n H {\umF_2}_\star \Omega^\rho S^{\rho + 1}$ occurs by construction in $H {\umF_2}_\star F_n \Omega^\rho S^{\rho + 1}.$ Thus it is sufficient to show $F_n H {\umF_2}_\star \Omega^\rho S^{\rho + 1}$ is a basis for $H {\umF_2}_\star F_n \Omega^\rho S^{\rho + 1}.$ 

 We show $\{\Phi^e (t^\epsilon e^k)\}$ forms a basis of $H {\mF_2}_* (F_n \Omega^\rho S^{\rho + 1})^e$ and $\{\Phi^{C_2} (t^\epsilon e^k)\}$ forms a basis of $H {\mF_2}_* (F_n \Omega^\rho S^{\rho + 1})^{C_2},$ which by \cite[Lemma 2.8]{BW2018} completes the proof. 
 
 In \cite[\textsection 4]{BW2018}, Behrens--Wilson compute 
 \[
 \Phi^e (t^\epsilon e^k) = x_1^{2k_1 + \epsilon_0} x_1^{2k_2 + \epsilon_1} \cdots.
 \]
 From this we see $\{\Phi^e (t^\epsilon e^k)\}$ forms a basis of $H {\mF_2}_* (F_n \Omega^\rho S^{\rho + 1})^e$ as computed in \cite[p. 239]{CohenLadaMay76}.

 Similarly, we identify 
 \[
 (F_n(\Omega^\rho S^{\rho + 1}))^{C_2} \simeq (F_n \Omega S^2) \times (F_n \Omega^2 S^3)
 \]
 by noticing that  
 \[
 C_\rho^k (S^1)^{C_2} \simeq \underset{i + 2j = k}{\coprod} \left(C_i(\mR) \underset{\Sigma_i}{\times}{(S^1)^{\times i}}\right) \times \left(C_j(\mR^2) \underset{\Sigma_j}{\times} (S^1)^{\times j} \right),
 \]
where $C_k(M)$ is an ordered configuration of $k$ distinct points in $M.$
 
Then using \cite[\textsection 4]{BW2018} identification of $\{\Phi^{C_2} (t^\epsilon e^k)\}$ and \cite[p. 239]{CohenLadaMay76}, we see  $\{\Phi^{C_2} (t^\epsilon e^k)\}$ indeed forms a basis of $H {\mF_2}_* (F_n \Omega^\rho S^{\rho + 1})^{C_2}.$
\end{proof}
	
Recall $\Omega^\rho S^{\rho + 1} \langle \rho + 1 \rangle$ denotes the homotopy fiber of the map $\Omega^\rho S^{\rho + 1} \xrightarrow{r} S^1 ,$  the $\rho$-loops of the unit map $S^{\rho + 1} \to K( \umZ, \rho + 1),$ so
$$\Omega^\rho S^{\rho + 1} \langle \rho + 1 \rangle \to \Omega^\rho S^{\rho + 1} \xrightarrow{r} S^1 $$
 is a fiber sequence.
The fibration splits if there is an epimorphism
$$\pi_{1}^{C_2} \Omega^\rho S^{\rho + 1} \to \pi_1^{C_2} S^1.$$

 \begin{lemma}
     \label{claim:wtTwoDiv}
     There is an epimorphism $\pi_{1}^{C_2} \Omega^\rho S^{\rho + 1} \to \pi_1^{C_2} S^1.$ Hence, $\Omega^\rho S^{\rho + 1} \simeq S^1 \times \Omega^\rho S^{\rho + 1} \langle \rho + 1 \rangle$ and $H {\umF_2}_\star (\Omega^\rho S^{\rho + 1} \langle \rho + 1 \rangle ) \subseteq H {\umF_2}_\star (\Omega^\rho S^{\rho + 1})$ is the span of monomials of weight divisible by 2. 
 \end{lemma}
\begin{proof}
	Stably, $\upi_{1} (\Omega^\rho S^{\rho + 1}) \cong \upi_{\rho + 1} S^{\rho + 1}$ is the Mackey functor 
	\[
	\begin{tikzcd}
		\mZ^2 \arrow[d,bend right =20,swap,"(1 \, \,  2)"]   \\
		\mZ \arrow[u, bend right=20].{} 
	\end{tikzcd}
	\] The induced map $\upi_1(r) : \upi^s_1 \Omega^\rho S^{\rho + 1} \to \upi_1 K(\umZ, 1)$ must be an epimorphism since the diagram of Mackey functors
	\[
	\begin{tikzcd}
		\mZ^2 \arrow[d,bend right = 20,swap,"(1 \, \, 2)"] \arrow[r] & \mZ \arrow[d, bend right = 20, swap, "\cong"]  \\
		\mZ \arrow[u, bend right = 20]{} \arrow[r, "\cong"] &  \mZ \arrow[u, bend right = 20] 
	\end{tikzcd}\] 
commutes. Evaluation at the $C_2 / C_2$ level gives an epimorphism 
$$\pi_1^{C_2} \Omega^\rho S^{\rho + 1} \to \pi_1^{C_2} S^1.$$ 
  \end{proof}

\subsection{Thom spectra of stable spherical fibrations} 
We describe a Lewis--May Thom spectrum functor for $H\umF_p$-local stable spherical fibrations $\mu: X \to BF_p$ (see \cite{LewMayStein86}, see also \cite[\textsection 4]{HW2020} and \cite[\textsection 3.4]{BlubCohSchli2010}). In the language of structured ring spectra, $BF_p$ can be identified with $BGL_1(\mS_p),$ the classifying space of the units of the $H \umF_2$-localized sphere spectrum $\mS_p.$

Let $V$ be a finite dimensional real $C_2$-representation and $F_p(V)$ be the topological monoid of basepoint-preserving $C_2$-equivariant homotopy equivalences of $S^V_p,$ the $H\umF _p$-localization of the one-point compactification of $V.$ There is an associated quasi-fibration
$$ EF_p(V) \to BF_p(V)$$
with fiber $S^V_p.$ Write $BF_p$ for the colimit of the spaces $BF_p(V)$ where the colimit is taken over a diagram with one object for each finite dimensional real $C_2$-representation and one arrow $V \to W$ if and only if $V \subset W.$

Given a map $f: X \to BF_p,$ let $X(V)$ be the closed subset $f^{-1} BF_p(V)$ and 
$$T(f)(V): = f^* EF_p(V)/X $$
be the Thom space of the induced map $f_V: X(V) \to BF(V).$ Here $f^*EF_p(V)$ is the pullback of $EF_p(V)$ and $X$ is viewed as a subspace via the induced section. 

This defines an object $T(f)$ in the category of $C_2$-orthogonal prespectra. Composition with the spectrification functor gives a Thom spectrum functor 
$$T_{\mS_p}: \mathcal{U} / BF_p \to \mathcal{S}^{C_2}_{p}$$
defined on the category $\mathcal{U} / BF_p$ of $C_2$-spaces over the classifying space $BF_p$ for $H\umF_p$-local stable spherical fibrations, with values in $H \umF_p$-local $C_2$-spectra $\mathcal{S}^{C_2}_{p}.$

\subsection{A Thom spectrum model for $\mathbf{H \umF_2}$} \label{sec:ThomModF2} 

Consider the $\rho$-loops of the unit map $S^{\rho + 1} \to K(\mZ, \rho + 1).$ Let $\mu$ denote composition with the adjoint to $-1 \in \pi_0^{C_2} (\mS_2^0):$ 
    $$ \mu: \Omega^\rho S^{\rho + 1} \to \Omega^\rho K(\umZ, \rho + 1) \to BG L_1(\mS_2^0).$$
\begin{lemma}
     \label{obs:HF2Thom}
The Thom class $(\Omega^\rho S^{\rho + 1})^\mu \to H \umF_2$ is an equivalence of $C_2$-spectra.
\end{lemma}
This follows from \cite[Proof of Theorem A]{Levy22}.

\section{A Thom spectrum model for $H \umZ_2$} \label{sec:ThomMod}
Arguing as in Hahn--Wilson \cite[Theorem~9.1]{HW2020}, which uses arguments of Antol\'{i}n-Camarena-Barthel \cite[\textsection~5.2]{AnCamBarth2019}, we prove 
\begin{theorem} \label{thm:HZcomp}
       There is an equivalence of $C_2$-spectra
    \begin{align*}
        (\Omega^\rho S^{\rho + 1} \langle \rho + 1 \rangle_2)^\mu \to H \umZ_2. 
    \end{align*}
    \end{theorem}

\begin{proof}
We first show the Thom class
\begin{align} \label{eq:ThomCl} 
        (\Omega^\rho S^{\rho + 1} \langle \rho + 1 \rangle)^\mu \to H \umZ_2 
\end{align}
 is an equivalance of $C_2$-spectra.   
    Decomposing $S^1$ into a $0$-cell and a $1$-cell and trivializing the fiber on each cell produces a decomposition of the Thom spectrum $(\Omega^\rho S^{\rho + 1})^\mu$ as a cofiber
$$ (\Omega^\rho S^{\rho + 1} \langle \rho + 1 \rangle)^\mu \xrightarrow{x} (\Omega^\rho S^{\rho + 1} \langle \rho + 1 \rangle)^\mu \to (\Omega^\rho S^{\rho + 1})^\mu \simeq H \umF_2. $$

Each of these Thom spectra come from bundles classified by $\mathbb{A}_2$-maps, which is enough to ensure the map $x$ induces a map 
$$\upi_* (\Omega^\rho S^{\rho + 1} \langle \rho + 1 \rangle )^\mu \to \upi_* (\Omega^\rho S^{\rho + 1} \langle \rho + 1 \rangle )^\mu$$ of modules over $\upi_0 (\Omega^\rho S^{\rho + 1} \langle \rho + 1 \rangle )^\mu.$ In particular, on homotopy the map corresponds to multiplication by some element $x \in \upi_0 (\Omega^\rho S^{\rho + 1} \langle \rho + 1 \rangle )^\mu.$

Taking loops of $\mu: \Omega^\rho S^{\rho + 1} \langle \rho + 1 \rangle \to BGL_1 (\mS_2^0),$ we obtain a map
$$ \Omega \Omega^\rho S^{\rho + 1} \langle \rho + 1 \rangle \to GL_1 (\mS_2^0). $$
By the universal property of $GL_1(\mS_2^0),$ we equivalently have an $E_1$-ring map
$$ f: \Sigma_+^{\infty} \Omega \Omega^\rho S^{\rho + 1} \langle \rho + 1 \rangle \to \mS_2^0.$$
Similarly, there is also a map $\epsilon: \Sigma_+^\infty \Omega \Omega^\rho S^{\rho + 1} \langle \rho + 1 \rangle \to \mS^0_2$ coming from the trivial map $\Omega^\rho S^{\rho + 1} \langle \rho + 1 \rangle \to GL_1(\mS_2^0).$

The construction of Thom spectra as bar constructions (see Definition 4.1 of \cite{AndoBlumGepHop2014}) then implies $\underline{\pi}_0 (\Omega^\rho S^{\rho + 1} \langle \rho + 1 \rangle )^\mu \cong \text{coker} \underline{\pi}_0 (f - \epsilon).$ 

Thus  $\underline{\pi}_0 (\Omega^\rho S^{\rho + 1} \langle \rho + 1 \rangle )^\mu$ is 
    \[
\underline{\pi}_0 (S^0_2) \quad = \quad
	\begin{tikzcd} 
            1 \arrow[|->]{d} 
            & t \arrow[|->]{d}
	    &\umZ_2 [t] / (t^2) \arrow[out=225, looseness=1, swap]{d}{res} 
            & t \\
	      1
            & 2
            &\umZ_2 \arrow[in=315, looseness=1, swap]{u}{tr} 
            & 1 \arrow[|->]{u}
	\end{tikzcd}
 \]
    modulo classes in the image of $f - \epsilon.$ This also fits into a short exact sequence
    $$ \underline{\pi}_0 \Omega^\rho S^{\rho + 1} \langle \rho + 1 \rangle)^\mu \xrightarrow{x}  \underline{\pi}_0 \Omega^\rho S^{\rho + 1} \langle \rho + 1 \rangle)^\mu  \to \umF_2. $$
    From the $C_2/e$ spot we deduce $x = 2$ and from the short exact sequence and Mackey functor structure deduce that 
    $$ \underline{\pi_0} (\Omega^\rho S^\rho + 1 \langle \rho + 1 \rangle)^\mu = \umZ_2. $$
    
    We have checked that the Thom class (\ref{eq:ThomCl}) induces an isomorphism on $\underline{\pi_0}$. To show that it is an equivalence of $C_2$-spectra, it remains to check that this map induces an isomorphism in $\underline{\pi}_V$ for $V \neq 0$, that is, $\underline{\pi}_V( (\Omega^\rho S^{\rho + 1} \langle \rho + 1 \rangle)^\mu)$ is trivial. 
    
    The underlying level follows from the classical nonequivariant result. 

    The genuine fixed point level follows from Nakayama's lemma once one argues that the genuine fixed point spectra have finitely generated homotopy groups in each degree. 
The isotropy separation reduces us to the corresponding statement on geometric fixed points.

The geometric fixed points of a Thom spectrum are given by the Thom spectrum of the fixed points, so
\[
H_* (\Omega^\rho S^{\rho + 1} \langle \rho + 1 \rangle^\mu)^{\Phi C_2} \cong H_*( \Omega^\rho S^{\rho + 1} \langle \rho + 1 \rangle^{C_2} )^\mu
.\]
We can apply the non-equivariant $H \mF_2$-Thom isomorphism to get
\[
H_*(\Omega^\rho S^{\rho + 1} \langle \rho + 1 \rangle^{C_2} )^\mu \cong H_* (\Omega^\rho S^{\rho + 1} \langle \rho + 1 \rangle^{C_2}).
\]

    Recall from Lemma \ref{claim:wtTwoDiv} that $\Omega^\rho S^{\rho + 1} \simeq S^1 \times \Omega^\rho S^{\rho + 1} \langle \rho + 1 \rangle$, and from Equation (\ref{eq:fixedpts}) that $\Omega^2 S^3 \times \Omega S^2 \simeq (\Omega^\rho S^{\rho + 1})^{C_2}$. It follows that 
    $$\Omega^2 S^3 \times \Omega S^2 \simeq (\Omega^\rho S^{\rho + 1})^{C_2} \simeq (S^1 \times \Omega^\rho S^{\rho + 1} \langle \rho + 1 \rangle) ^{C_2} \simeq S^1 \times (\Omega^\rho S^{\rho + 1} \langle \rho + 1 \rangle)^{C_2}.$$
    Thus $H_* (\Omega^\rho S^{\rho + 1} \langle \rho + 1 \rangle)^{C_2} \subseteq H_* (\Omega^2 S^3 \times \Omega S^2)$ has finitely generated homology groups in each degree. Hence, it also has finitely generated homotopy groups in each degree, and we can apply Nakayama's lemma to finish the proof of the genuine fixed point level. Therefore $(\Omega^\rho S^{\rho + 1} \langle \rho + 1 \rangle)^\mu \simeq H \umZ_2$.

Recalling that the last map in the composition 
\[
\mu: \Omega^\rho S^{\rho + 1} \langle \rho + 1 \rangle \to \Omega^\rho S^{\rho + 1}\to S^1 \to BGL_1(\mS_2^0).
\]
is defined to be the adjoint to $-1 \in \pi_0^{C_2} (\mS_2^0)$ in Equation (\ref{eq:bundle}),
we define $\mu_2$ by taking the $H \umF_2$-localization
\begin{center}
\begin{tikzcd}
    \mu: \Omega^\rho S^{\rho + 1} \langle \rho + 1 \rangle \ar[r] \ar[d] & \Omega^\rho S^{\rho + 1} \ar[r] \ar[d] & S^1 \ar[r] \ar[d] & BGL_1(\mS_2^0). \\
    \mu_2: \Omega^\rho S^{\rho + 1} \langle \rho + 1 \rangle_2 \ar[r] & \Omega^\rho S^{\rho + 1}_2 \ar[r] & S^1_2 \ar[ur] &
\end{tikzcd}
\end{center}

Thus the localization map on the base spaces induces a map
\[
f: \Omega^\rho S^{\rho + 1} \langle \rho + 1 \rangle^\mu \to \Omega^\rho S^{\rho + 1} \langle \rho + 1 \rangle_2^{\mu_2}. 
\]
Since our Thom construction takes values in $H \umF_2$-local $C_2$-spectra, both \\ $\Omega^\rho S^{\rho + 1} \langle \rho + 1 \rangle^\mu$ and $\Omega^\rho S^{\rho + 1} \langle \rho + 1 \rangle_2^{\mu_2}$ are $H \umF_2$-local. Thus it suffices to show that $f$ induces an isomorphism on $H \umF_2$-homology. 

The localization map on the base spaces induces an $H \umF_2$-isomorphism by definition. They also each have an $H \umF_2$-Thom isomorphism since the composites 
\[
S^1 \to BGL_1 (\mS^0) \to BGL_1 (H \umF_2) \]
\[
S^1_2 \to BGL_1 (\mS^0_2) \to BGL_1 (H \umF_2)
\]
are null as the maps $\mu$ and $\mu_2$ represent the class $-1 \in \pi_0^{C_2} GL_1 (\umF_2)$, but $-1 = 1$ is the basepoint. 
Thus $f$ induces an isomorphism on $H \umF_2$-homology and hence is an equivalence. Therefore, slightly abusing notation by replacing $\mu_2$ with $\mu,$ $(\Omega^\rho S^{\rho + 1} \langle \rho + 1 \rangle_2)^\mu \simeq H \umZ_2$.
\end{proof}
 
\section{A product fibration} \label{sec:fib} 
Recall $X_2$ denotes the Bousfield localization of $X$ with respect to $H {\umF_2},$ we have set
$$\cF_n = (F_n \Omega^\rho S^{\rho + 1})_2,$$
and there are product maps 
	$$ \cF_m \times \cF_n \overset{\mu}{\rightarrow} \cF_{m + n} .$$
 Furthermore $A_n$ is defined by the homotopy fiber sequence 
	\begin{align} 
		A_n \to \cF_{2n + 1} \to S^1_2
	\end{align}
	where the second map is the localization of the composite
	$$ F_{2n + 1} \Omega^\rho S^{\rho + 1} \to \Omega^\rho S^{\rho + 1} \to S^1. $$ 
	
	We prove the $C_2$-equivariant analogue of \cite[Theorem~1.3]{CDGM1988}.
	\begin{theorem} \label{thm:main}
		The fiber sequence (\ref{eq:fib}) is equivalent to a product fibration. Indeed, there is a $C_2$-equivariant map $A_n \overset{\phi}{\rightarrow} \cF_{2n}$ and a commutative diagram of fibrations
		\[
		\begin{tikzcd}
			A_n \arrow[rr] \arrow[d]{} & & A_n \arrow[d] \\
			S_2^1 \times A_n \arrow[r, "1 \times \phi"] \arrow[d, "p_1"] & S_2^1 \times \cF_{2n}= \cF_1 \times \cF_{2n} \arrow[r, "m"] & \cF_{2n + 1} \arrow[d] \\
			S_2^1 \arrow[rr, equal] & & S_2^1
		\end{tikzcd}
		\]
		which is an equivalence on total spaces and on fibers.
	\end{theorem}
 
Our proof of Theorem \ref{thm:main} is a $C_2$-equivariant analogue of \cite[\textsection 2. Proof of Theorem~1.3]{CDGM1988}. Following \cite{CDGM1988}, we first show that the inclusion 
$$F_{m - 1} \Omega^\rho S^{\rho + 1} \to F_m \Omega^\rho S^{\rho + 1}$$ may be considered as a $C_2$-equivariant inclusion into a mapping cone. For a pointed $C_2 \times \Sigma_m$-space $X,$ we define
	$$ M_m(X) = C_m(\rho) \underset{\Sigma_m}{\times} X / (C_m(\rho) \underset{\Sigma_m}{\times}{*}). $$
	Let $I$ denote the unit interval, $I^m$ the $m$-dimensional unit interval, and $\partial I^m$ the boundary of $I^m,$ all with trivial $C_2$-action. Note that $I / \partial I \simeq S^1$ and $\partial I^m \simeq S^{m - 1}.$
	
	\begin{lemma} \label{lemma}
		Let $F_m = F_m \Omega^\rho S^{\rho + 1}.$ There is a cofibration sequence
		$$ M_m (\partial I^m) \xrightarrow[]{c} F_{m - 1} \to F_m. $$
	\end{lemma}
	
	\begin{proof}
		Let $T^m(S^1)$ denote the wedge with trivial $C_2$-action consisting of points in the $m$-fold Cartesian product having at least one component the basepoint. Viewing $I^m$ as the cone of the natural map $\partial I^m \to T^m (I / \partial I)$ yields a $C_2 \times \Sigma_m$-equivariant cofibration
		$$ \partial I^m \to T^m (I / \partial I) \hookrightarrow (I / \partial I)^{\times m}, $$
		with trivial $C_2$-action and hence a cofibration 
		\begin{align} \label{eq:cofib}
			M_m(\partial I^m) \xrightarrow[]{k} M_m (T^m(I / \partial I)) \to M_m((I/ \partial I)^{\times m}. 
		\end{align}
		
		Recall from \cite{RS200} that
		$$ F_m = \left( \bigcup_{k \leq m} M_k ((I / \partial I)^{\times k} ) \right) / \sim. $$
		
		The map $c$ of Lemma \ref{lemma} is the composite
		$$ M_m(\partial I^m) \to M_m (T^m (I / \partial I)) \to F_{m - 1}, $$
		where the second map uses the equivalence relation from (\ref{eq:filtration}) to ignore the basepoint in at least one component. The required homeomorphism from the mapping cone of $c$ to $F_m$ is a quotient of the homeomorphism in (\ref{eq:cofib}) from the mapping cone of $k$ to $M_m((I / \partial I)^{\times m}.$
	\end{proof}
	
	We now return to the proof of Theorem \ref{thm:main}. Assume by induction that the theorem has been proved for $n - 1$ so $\cF_{2(n - 1) + 1} \simeq S_2^1 \times A_{n - 1}.$ Note that $\cF_{2(n - 1) + 1} \to \cF_{2n - 1}$ is an equivalence. Localizing Lemma \ref{lemma} yields a map
	\begin{align} \label{eq:partial} M_{2n} (\partial I^{2n})_2 \to \cF_{2n - 1} \simeq S_2^1 \times A_{n - 1}.
        \end{align}

	\begin{lemma}
	   The cohomology $H^1_{C_2} (M_{2n}(\partial I^{2n})_2; {\umZ_2} \,) = 0,$ so the map (\ref{eq:partial}) is of the form $* \times h$, for some map $h: M_{2n}(\partial I^{2n})_{2} \rightarrow A_{n-1}$.  
	\end{lemma}
	\begin{proof}
            The freeness of the action of $\Sigma_{2n}$ on $C_{2n}(\rho)$ implies
 \[C_{2n}(\rho) \underset{\Sigma_{2n}}{\times} S^{2n-1} / (C_{2n} (\rho) \underset{\Sigma_{2n}}{\times} *) \simeq C_{2n} (\rho)/ \Sigma_{2n} \times S^{2n-1}/(C_{2n} (\rho) / \Sigma_{2n} \times *) .\] Since $S^{2n-1}$ is $2n-1$-connective, we can conclude that $H^{1}(M_{2n}(\partial I^{2n}); \umZ_2) = 0$.
\end{proof}

 We return to the proof of Theorem \ref{thm:main}. Let $Y$ denote the mapping cone of $h.$ The map of cofibrations
	\[
	\begin{tikzcd}
		M_{2n} (\partial I^{2n})_2 \arrow[r]{h} \arrow[d, equal] & A_{n - 1} \arrow[r] \arrow[d] & Y \arrow[r] \arrow[d, "\phi"] & \cF_{2n} / \cF_{2n - 1} \arrow[d, equal] \\
		M_{2n}(\partial I^{2n})_2 \arrow[r] & \cF_{2n - 1} \arrow[r] & \cF_{2n} \arrow[r] & \cF_{2n} / \cF_{2n - 1}
	\end{tikzcd}
	\]
	shows that $H {\umF_2}_\star Y \to H {\umF_2}_\star \cF_{2n}$ is injective with image spanned by monomials of weight divisible by $2.$ There is a map of fibrations 
	\[
	\begin{tikzcd}
		Y \arrow[rr] \arrow[d] & & A_n \arrow[d] \\
		S_2^1 \times Y \arrow[r, "1 \times \phi"] \arrow[d, "p_1"] & S_2^1 \times \cF_{2n}= \cF_1 \times \cF_{2n} \arrow[r, "\mu"] & \cF_{2n + 1} \arrow[d] \\
		S_2^1 \arrow[rr, equal] & & S_2^1.
	\end{tikzcd}
	\]
	The map of total spaces induces an isomorphism in $H \umF_2$-homology, and since $\cF_{2n + 1}$ is $H \umF_2$-local, this map is a $H \umF_2$-localization. Hence, there is an equivalence of fibrations 
	\[
	\begin{tikzcd}
		Y_p \arrow[r] \arrow[d] & A_n \arrow[d] \\
		S_2^1 \times Y_p \arrow[r]{} \arrow[d]{} & \cF_{2n + 1} \arrow[d] \\
		S_2^1 \arrow[r, equal] & S_2^1, 
	\end{tikzcd}
	\]
	extending the induction and completing the proof of Theorem \ref{thm:main}.

Since $\phi: Y_{p} \rightarrow \cF_{2n}$ induces an injection $H {\umF_2}_\star Y \to H {\umF_2}_\star \cF_{2n}$ with image spanned by monomials of weight divisible by $2$, the following corollary is immediate.
 \begin{cor} \label{cor:monomialwt}
    The map $H {\umF_2}_\star A_n \rightarrow H {\umF_2}_\star\cF_{2n + 1 }$ is the inclusion of monomials of weight divisible by $2$ into monomials of weight less than or equal to $2n + 1.$
 \end{cor}

 \section{Proof of Main Theorem}

Now we can use Theorems \ref{thm:HZcomp} and \ref{thm:main} to construct $B_0(n),$ a $C_2$-equivariant analogue of the $n$th integral Brown-Gilter spectrum at the prime $2$ as a Thom spectrum.
Let $\tilde{A_n}$ denote the homotopy fiber of the composite
 $$F_{2n + 1} \Omega^\rho S^{\rho + 1} \to \Omega^\rho S^{\rho + 1} \to S^1$$
 so the $H \umF_2$-localization of $\tilde{A_n}$ is $A_n.$
Using the $H \umF_2$-localization of the commutative diagram
	\begin{equation} \label{eq:comdiag}
		\begin{tikzcd}
			\tilde{A_n} \arrow[r] \arrow[d, "i_n"]{} & F_{2n + 1} \arrow[d] \\
			\Omega^\rho S^{\rho + 1} \langle \rho + 1 \rangle \arrow[r] & \Omega^\rho S^{\rho + 1},
		\end{tikzcd}
	\end{equation}
we define $B_0(n)$ to be the Thom spectrum $(A_n)^{{\mu}}.$ 
	
	Taking the Thom spectrum of the composites
	$$ A_m \times A_n \xrightarrow[]{\phi_m \times \phi_n} \cF_{2m} \times \cF_{2n} \to \cF_{2m + 2n} \to \cF_{2n + 2m + 1} \to A_{m + n},$$
	where the last map splits the equivalence of Theorem \ref{thm:main}, yields pairings 
 $$B_0(m) \wedge B_0(n) \to B_0( m + n).$$

Combining Corollary \ref{cor:monomialwt} with Diagram (\ref{eq:comdiag}), we see that the map ${(i_n)}^{{\mu}} : B_0(n) \to H \umZ_2$ induces a monomorphism in homology, with image spanned by monomials of weight less than or equal to $2n$. We have now proven Theorem \ref{thm:C2intBGspectra}, our $C_2$-equivariant analogue of \cite[Theorem~1.5(i),(ii)]{CDGM1988}.
	\begin{theorem}\label{thm:C2intBGspectra}
		For $n > 0,$ there is a $2$-complete spectrum $B_0(n)$ and a map 
		$$B_0(n) \xrightarrow[]{g} H {\umZ_{2}} $$
		such that
		\begin{itemize}
			\item[(i)] $g_\star$ sends $H {\umF_2}_\star(B_0(n))$ isomorphically onto the span of monomials of weight $\leq 2n;$
			\item[(ii)] there are pairings
			$$ B_0(m) \wedge B_0(n) \to B_0 (m + n) $$
			whose homology homomorphism is compatible with the multiplication in $H {\umF_2}_\star ( H \umZ_2 ).$
		\end{itemize}
	\end{theorem}

	\begin{rmk} \label{rmk:GeomCond}
		It is not currently known if a $C_2$-equivariant analogue of \cite[Theorem~1.5(iii)]{CDGM1988} holds or if there should be some other criterion determining (integral) Brown--Gitler spectra in the $C_2$-equivariant case. In the non-equivariant case, \cite[Theorem~1.5(iii)]{CDGM1988} states that for any CW-complex $X,$
		$$g_*: B_0(n)_i(X) \to H_i(X;\mZ_2)$$
		is surjective if $i \leq 2p(n + 1) - 1.$ This condition originated in the obstruction theoretic construction
  and would be interesting to study further in the $C_2$-equivariant setting. 
\end{rmk}

\bibliographystyle{alpha}
\bibliography{bib}

\end{document}

%% file: preamble.tex
\allowdisplaybreaks[1]

\usepackage{color}
\usepackage{tikz-cd}
\usetikzlibrary{arrows}
\usepackage{amsmath,amssymb}
\usepackage{graphicx}
\usepackage{mathrsfs}
\usepackage{mathabx}
\usepackage[all]{xy}
\usepackage{adjustbox}
\usepackage[normalem]{ulem}

\newtheorem{theorem}{Theorem}[section]
\newtheorem{cor}[equation]{Corollary}
\newtheorem{lemma}[equation]{Lemma}
\newtheorem{proposition}[equation]{Proposition}

\newtheorem*{thm*}{Theorem}

\newtheorem*{cor*}{Corollary}
\newtheorem*{lem*}{Lemma}
\newtheorem*{prop*}{Proposition}

\theoremstyle{definition}

\newtheorem{rmk}[equation]{Remark}

\newtheorem*{defn*}{Definition}
\newtheorem*{ex*}{Example}
\newtheorem*{exs*}{Examples}
\newtheorem*{rmk*}{Remark}
\newtheorem*{claim*}{Claim}

\numberwithin{equation}{section}
\numberwithin{figure}{section}

\DeclareMathOperator*{\holim}{holim}

\DeclareMathOperator{\wt}{wt}

\newcommand{\mF}{{\mathbb F}}

\newcommand{\mR}{{\mathbb R}}
\newcommand{\mS}{{\mathbb S}}

\newcommand{\mZ}{{\mathbb Z}}

\newcommand{\cA}{{\mathcal A}}

\newcommand{\cF}{{\mathcal F}}

\newcommand{\umF}{\underline{\mF}}

\newcommand{\umZ}{\underline{\mZ}}

\newcommand{\upi}{\underline{\pi}}

\newcommand{\bxi}{\bar{\xi}}